\documentclass[a4paper,11pt]{amsart}
\usepackage{amssymb}
\usepackage{latexsym}
\usepackage{amsmath}
\usepackage{enumerate}
\usepackage{amsmath, hyperref}

\usepackage{tikz}
\usepackage{geometry}

\newtheorem{theorem}{Theorem}[section]
\newtheorem{lemma}[theorem]{Lemma}

\theoremstyle{definition}
\newtheorem{definition}[theorem]{Definition}

\newtheorem{fact}[theorem]{Fact}

\newtheorem{remark}[theorem]{Remark}

\newtheorem*{claim}{Claim}

\title{There are no minimal effectively inseparable theories}
\author{Yong Cheng}
\address{School of Philosophy, Wuhan University, China}
\thanks{We thank Albert Visser for comments on the preprint version of this paper, and thank helpful comments  for improvements from the referees.}

\subjclass[2010]{03F40, 03F25, 03F30}

\keywords{Effective inseparability, Essential undecidability, Interpretation}

\begin{document}

\setcounter{tocdepth}{4}
\setcounter{page}{0}

\thispagestyle{empty}
\newpage

\begin{abstract}
This paper belongs to the research on the limit of the first incompleteness theorem. Effectively inseparable theories ({\sf EI}) can be viewed as an effective version of essentially undecidable theories ({\sf EU}), and {\sf EI} is stronger than  {\sf EU}.
We examine the question: are  there  minimal effectively inseparable theories with respect to interpretability. We propose {\sf tEI}, the theory version of {\sf EI}. We first prove that there are no minimal {\sf tEI} theories with respect to interpretability (i.e., for any {\sf tEI} theory $T$, we can effectively find a theory which is {\sf tEI} and strictly weaker than $T$ with respect to interpretability). By a theorem due to Marian B. Pour-EI,  we have {\sf tEI} is equivalent with {\sf EI}. Thus, there are no minimal {\sf EI} theories with respect to interpretability.  Also we prove that there are no minimal finitely axiomatizable {\sf EI} theories with respect to  interpretability.
\end{abstract}

\maketitle

\section{Introduction}

Since G\"{o}del, research on incompleteness has greatly deepened our understanding of the incompleteness phenomenon. The motivation of this work is to explore the limit of the first incompleteness theorem.
We can understand incompleteness in an abstract way via meta-mathematical
properties of formal theories which exhibit behaviors that can be related to incompleteness/undecidability.
For example, a general form of G\"{o}del's first incompleteness theorem says that any consistent recursively enumerable (RE) extension of Robinson Arithmetic $\mathbf{Q}$ (see Definition \ref{def of Q}) is essentially undecidable ({\sf EU}).

Meta-mathematical properties of RE   theories we discuss in this paper  include essentially incomplete, essentially undecidable, {\sf Creative}, effectively inseparable ({\sf EI}) and effectively extensible ({\sf EET}) theories. For the definitions of these theories, see Definition \ref{def of theories}.
We give an overview of the relationships between these notions. From Lemma \ref{relation about EI}, essentially incomplete theories  are equivalent with essentially undecidable theories. From Theorem \ref{Symullyan thm}, {\sf EET} theories  are equivalent with {\sf EI} theories. From Lemma \ref{relation about EI}, an {\sf EI} theory is {\sf Creative}, but a {\sf Creative} theory may not be {\sf EI}; an {\sf EU} theory may not be {\sf Creative}, and a {\sf Creative} theory may not be {\sf EU}; the notion of {\sf EI} is stronger than {\sf EU}: an {\sf EI} theory is  {\sf EU}, but an {\sf EU} theory may not be {\sf EI}.
Typical examples of {\sf EI} theories are Robinson Arithmetic $\mathbf{Q}$  and the theory $\mathbf{R}$ (see Definition \ref{def of R}).

Given a meta-mathematical
property related to incompleteness/undecidability, a natural question is: are there minimal theories with this property? For essentially undecidable theories, \cite{PV}  proves that there are no minimal essentially undecidable theories with respect to (w.r.t. for short) interpretability.

Effective inseparability is an important meta-mathematical property of RE theories: (1)
effective inseparability can be viewed as an  effective version of essential undecidability; (2) recursion-theoretic proofs of metamathematical results tend to rely on an effectively inseparable  pair of RE sets and its properties.
In this work, we examine the question: are there  minimal effectively inseparable theories?
The answer of this question depends on how we define the notion of minimality.
If we view a theory as minimal if it has a minimal number of axioms, then Robinson Arithmetic $\mathbf{Q}$ is a minimal {\sf EI} theory since $\mathbf{Q}$ is finitely axiomatized. For a theory which is not finitely axiomatizable, if we view it as minimal if it has a minimal number of axiom schemes, then the Vaught set theory $\mathbf{VS}$ (see Definition \ref{def of VS}) is a minimal {\sf EI} theory (see Fact \ref{fact on VS}) since $\mathbf{VS}$ has only one axiom scheme.
If we view an {\sf EI} theory as minimal if, after deleting one of its axioms, the remaining theory is no longer {\sf EI}, then $\mathbf{Q}$ is a minimal {\sf EI} theory since it is minimal effectively undecidable: if an axiom of $\mathbf{Q}$ is deleted, then the remaining theory is not {\sf EU} (\cite[p.62]{undecidable}).
If we view a theory as minimal if it has a minimal Turing degree, then all {\sf EI} theories are minimal since  any {\sf EI} theory  has Turing degree $\mathbf{0}^{\prime}$ by Lemma \ref{relation about EI}. Thus, it makes no sense to ask for minimal {\sf EI} theories w.r.t. Turing degree. In this paper, we examine the question whether there are  minimal {\sf EI} theories w.r.t. interpretability.

An effectively inseparable pair of RE sets is about sets of natural numbers.
In this paper, we introduce {\sf tEI} theories, the theory version of {\sf EI} theories, and prove the main theorem that  there are no minimal {\sf tEI} theories w.r.t. interpretability: for any {\sf tEI} theory $T$, we can effectively find a theory which is {\sf tEI} and strictly weaker than $T$ w.r.t. interpretability. By a theorem due to Marian B. Pour-EI,  we show that {\sf tEI} is equivalent with {\sf EI}. Thus, there are no minimal {\sf EI} theories w.r.t. interpretability.
Moreover, we show that there are no minimal finitely axiomatizable {\sf EI} theories w.r.t. interpretability. We give two proofs of this result.

If there is a minimal {\sf EI} theory, then such a theory is distinguished and could be viewed as a canonical theory of incompleteness. But we show that there are no minimal {\sf EI} theories w.r.t. interpretability. The construction of weak {\sf EI} theories  in this paper uses purely logical methods. The research on concrete incompleteness seeks for concrete true arithmetic statements from classical  mathematics which are not provable in $\mathbf{PA}$.
Both the research on meta-mathematics of arithmetic and the research on concrete incompleteness reveal that the incompleteness phenomenon is ubiquitous in both abstract formal theories and concrete mathematical theories.

This paper is structured as follows. In Section 2, we introduce basic notions and results we use in the paper. In Section 3, we prove that there are no minimal {\sf tEI} theories w.r.t. interpretability: for any {\sf tEI} theory $T$, we can effectively find a theory which is {\sf tEI} and strictly weaker than $T$ w.r.t. interpretability. By a theorem due to Marian B. Pour-EI,  we show that {\sf tEI} is equivalent with {\sf EI}. As a corollary, there are no minimal {\sf EI} theories w.r.t. interpretability.  In Section 4, we prove that there are no minimal finitely axiomatizable {\sf EI} theories w.r.t. interpretability.

\section{Preliminaries}

In this paper, we work with  first-order theories with finite signatures, and
all theories are supposed to be RE. We equate a theory with the set of theorems provable in it. We always assume the arithmetization of the base theory. Given a sentence $\phi$, let $\ulcorner\phi\urcorner$ denote the G\"{o}del number of $\phi$.
Under arithmetization, we equate a set  of sentences  with the set of G\"{o}del numbers of these sentences. Unless specifically stated,  recursive functions always mean partial recursive functions in this paper.
\smallskip

\begin{definition}[Basic notions]~\label{basic def}
\begin{enumerate}[(1)]
  \item We denote the RE set with index $i$ by $W_i$ where $W_i=\{x: \exists y \, T_1(i,x, y)\}$ and $T_1(z,x, y)$ is the Kleene predicate (see \cite{Kleene}).
  \item We  say that a pair $(A,B)$ of disjoint RE sets is  \emph{effectively inseparable} $(\sf EI)$ if there is a recursive function $f(x,y)$ such that for any $i$ and $j$, if $A\subseteq W_i$ and $B\subseteq W_j$ with $W_i\cap W_j=\emptyset$, then $f(i,j)$ converges and $f(i,j)\notin W_i\cup W_j$ (see \cite[p.94]{Rogers87}).
      \item We say that $A\subseteq\mathbb{N}$ is \emph{productive} if there exists a recursive function $f(x)$ (called a productive function for $A$) such that for every number $i$, if $W_i\subseteq A$, then $f(i)\in A- W_i$ (see \cite[p.84]{Rogers87}).
      \item We say that $A\subseteq\mathbb{N}$ is \emph{creative} if $A$ is RE and the complement of $A$ is productive (see \cite[p.84]{Rogers87}).
      \item We denote the recursive Turing degree by $\mathbf{0}$, and the jump or completion of $\mathbf{0}$ by $\mathbf{0}^{\prime}$ (see \cite[p.256]{Rogers87}).
\end{enumerate}
\end{definition}

Now we introduce the notions of essentially incomplete, essentially undecidable, {\sf Creative}, effectively inseparable and effectively extensible theories.
Essentially undecidable and essentially incomplete theories are introduced in \cite[p.14]{undecidable}.
Effectively inseparable  theories  are introduced in \cite[p.119]{Smullyan}.
Effectively extensible theories, an effective version of {\sf EU} theories, are introduced in \cite[Definition 9]{Pour-EI}.

\begin{definition}\label{def of theories}
Let  $T$ be a consistent RE theory.
\begin{enumerate}[(1)]
  \item We say $T$ is \emph{essentially incomplete} if  any consistent RE extension of $T$ over the same language is incomplete.
\item We say $T$ is \emph{essentially undecidable} ({\sf EU}) if any consistent RE extension of $T$ over the same language is undecidable.
\item Let $T_P$ be the set of G\"{o}del numbers of sentences provable in $T$ (i.e., $T_P=\{\ulcorner\phi\urcorner: T\vdash\phi\}$), and $T_R$ be the set of G\"{o}del numbers of sentences refutable in $T$ (i.e., $T_R=\{\ulcorner\phi\urcorner: T\vdash\neg\phi\}$).
      The pair $(T_P, T_R)$ is called the \emph{nuclei}  of the theory $T$.
    \item We say $T$ is \emph{{\sf Creative}} if $T_P$ is creative.
  \item We say $T$ is \emph{effectively inseparable}  ($\sf EI$) if $(T_P, T_R)$ is an $\sf EI$ pair.
  \item We say $T$ is \emph{effectively extensible}  ({\sf EET}) if there exists a recursive function $f$ such that if $i$ is the index of a consistent RE extension $S$ of $T$, then $f(i)$ outputs a sentence which is independent of $S$.\footnote{In \cite[Definition 9]{Pour-EI}, effectively extensible theories are defined based on a presentation of a theory which consists of a set of axioms and a set of reference rules of the theory. But as results in \cite{Pour-EI} show, the presentation involved is not essential for the proof of the main result in \cite{Pour-EI} that {\sf EET} is equivalent with {\sf EI}.} (see \cite[Definition 9]{Pour-EI})
\end{enumerate}
\end{definition}
\smallskip

\begin{fact}[\cite{Rogers87}, p.183, p.94]~\label{fact on creative}
\begin{enumerate}[(1)]
  \item Let $T$ be a consistent RE theory. The theory $T$ is {\sf Creative} iff any RE set is reducible to $T_P$: for any RE set $X$, there exists a recursive function $f$ such that $n\in X \Leftrightarrow f(n)\in T_P$.
  \item For any disjoint pair $(A,B)$ of RE sets, if $(A,B)$ is ${\sf EI}$, then both $A$ and $B$ are creative.
\end{enumerate}
\end{fact}
\smallskip

\begin{theorem}[\cite{Shoenfield61}, pp.172-173]~\label{Shoenfield thm}
There exists an essentially undecidable  theory  which is not {\sf Creative}.
\end{theorem}

Lemma \ref{relation about EI} and Theorem \ref{Symullyan thm} establish the relationships between essentially incomplete, {\sf EU}, {\sf Creative}, {\sf EI} and {\sf EET} theories.

\begin{lemma}\label{relation about EI}
Let $T$ be a consistent RE theory.
\begin{enumerate}[(1)]
  \item If  $T$ is ${\sf EI}$, then $T$ is ${\sf EU}$.
  \item If $T$ is ${\sf EI}$, then $T$ is {\sf Creative}.
  \item If $T$ is {\sf Creative}, then $T$ has Turing degree $\mathbf{0}^{\prime}$.
  \item $T$ is ${\sf EU}$ iff $T$ is essentially incomplete.
  \item ``$T$ is {\sf Creative}" does not imply ``$T$ is ${\sf EU}$".
  \item ``$T$ is ${\sf EU}$" does not imply ``$T$ is {\sf Creative}".
  \item ``$T$ is {\sf Creative}" does not imply ``$T$ is ${\sf EI}$".
\item  Any consistent RE extension of the theory $\mathbf{R}$ is {\sf EI}.
\end{enumerate}
\end{lemma}
\begin{proof}\label{}
\begin{enumerate}[(1)]
  \item Suppose  $T$ is ${\sf EI}$, but it  is not ${\sf EU}$. Let $S$ be a consistent RE extension of $T$ such that $S$ is decidable.  Suppose $(T_P, T_R)$ is ${\sf EI}$ via the recursive function $f$, and $S_P=W_i$ and $\overline{S_P}$, the complement of $S_P$, is $W_j$. Since $T_P\subseteq W_i$ and $T_R\subseteq W_j$, we have $f(i,j)$ converges and $f(i,j)\notin W_i\cup W_j=\mathbb{N}$, which is a contradiction.
  \item  Follows from Fact \ref{fact on creative}(2).
  \item   Follows from Fact \ref{fact on creative}(1).
\item  Follows from Theorem 2 in \cite[p.15]{undecidable}.
\item See Theorem 4.12 in \cite{Cheng23}.
\item Follows from Theorem \ref{Shoenfield thm}.
\item Follows from (5) since ${\sf EI}$ implies ${\sf EU}$.
\item Follows from the definitions and the fact that $\mathbf{R}$ is {\sf EI}.
\end{enumerate}
\end{proof}

\begin{theorem}[\cite{Pour-EI}, Theorem 1]~\label{Symullyan thm}
Let  $T$ be a consistent RE theory. Then $T$ is {\sf EET} iff $T$ is {\sf EI}.
\end{theorem}

Robinson Arithmetic $\mathbf{Q}$ and the theory $\mathbf{R}$ were introduced by Tarski, Mostowski and R.\ Robinson  in \cite[pp.51-53]{undecidable}, which are important base theories in the study of incompleteness and undecidability.\smallskip

\begin{definition}[Robinson Arithmetic $\mathbf{Q}$]~\label{def of Q}
Robinson Arithmetic $\mathbf{Q}$  is  defined in   the language $\{\mathbf{0}, \mathbf{S}, +, \times\}$ with the following axioms:
\begin{description}
  \item[$\mathbf{Q}_1$] $\forall x \forall y(\mathbf{S}x=\mathbf{S} y\rightarrow x=y)$;
  \item[$\mathbf{Q}_2$] $\forall x(\mathbf{S} x\neq \mathbf{0})$;
  \item[$\mathbf{Q}_3$] $\forall x(x\neq \mathbf{0}\rightarrow \exists y (x=\mathbf{S} y))$;
  \item[$\mathbf{Q}_4$]  $\forall x\forall y(x+ \mathbf{0}=x)$;
  \item[$\mathbf{Q}_5$] $\forall x\forall y(x+ \mathbf{S} y=\mathbf{S} (x+y))$;
  \item[$\mathbf{Q}_6$] $\forall x(x\times \mathbf{0}=\mathbf{0})$;
  \item[$\mathbf{Q}_7$] $\forall x\forall y(x\times \mathbf{S} y=x\times y +x)$.
\end{description}
\end{definition}
\begin{definition}\label{def of R}
Let $\mathbf{R}$ be the theory consisting of the following axiom schemes where $L(\mathbf{R})=\{\mathbf{0}, \mathbf{S}, +, \cdot, \leq\}$ and $x\leq y:=\exists z (z+x=y)$.
\begin{description}
  \item[\sf{Ax1}] $\overline{m}+\overline{n}=\overline{m+n}$;
  \item[\sf{Ax2}] $\overline{m}\cdot\overline{n}=\overline{m\cdot n}$;
  \item[\sf{Ax3}] $\overline{m}\neq\overline{n}$, if $m\neq n$;
  \item[\sf{Ax4}] $\forall x(x\leq \overline{n}\rightarrow x=\overline{0}\vee \cdots \vee x=\overline{n})$;
  \item[\sf{Ax5}] $\forall x(x\leq \overline{n}\vee \overline{n}\leq x)$.
\end{description}
\end{definition}

Now we introduce the notion of interpretability. \smallskip

\begin{definition}[Translations and interpretations, \cite{Visser16}, p.10-13]~
\begin{itemize}
\item We use $L(T)$ to denote the language of the theory $T$. Let $T$ be a theory in a language $L(T)$, and $S$ a theory in a language $L(S)$. In its simplest form, a \emph{translation} $I$ of language $L(T)$ into language $L(S)$ is specified by the following:
\begin{itemize}
  \item an $L(S)$-formula $\delta_I(x)$ denoting the domain of $I$;
  \item for each relation symbol $R$ of $L(T)$,  as well as the equality relation =, an $L(S)$-formula $R_I$ of the same arity;
  \item for each function symbol $F$ of $L(T)$ of arity $k$, an $L(S)$-formula $F_I$ of arity $k + 1$.
\end{itemize}
\item If $\phi$ is an $L(T)$-formula, its $I$-translation $\phi^I$ is an $L(S)$-formula constructed as follows: we rewrite the
formula in an equivalent way so that function symbols only occur in atomic subformulas of the
form $F(\overline{x}) = y$, where $\overline{x}, y$ are variables; then we replace each such atomic formula with $F_I(\overline{x}, y)$,
we replace each atomic formula of the form $R(\overline{x})$ with $R_I(\overline{x})$, and we restrict all quantifiers and
free variables to objects satisfying $\delta_I$. We take care to rename bound variables to avoid variable
capture during the process.
\item A translation $I$ of $L(T)$ into $L(S)$ is an \emph{interpretation} of $T$ in $S$ if $S$ proves the following:
\begin{itemize}
  \item for each function symbol $F$ of $L(T)$ of arity $k$, the formula expressing that $F_I$ is total on $\delta_I$:
\[\forall x_0, \cdots \forall x_{k-1} (\delta_I(x_0) \wedge \cdots \wedge \delta_I(x_{k-1}) \rightarrow \exists y (\delta_I(y) \wedge F_I(x_0, \cdots, x_{k-1}, y)));\]
  \item the $I$-translations of all theorems of $T$, and axioms of equality.
\end{itemize}
\end{itemize}
\end{definition}
The simplified picture of translations and interpretations above actually describes only \emph{one-dimensional}, \emph{parameter-free}, and \emph{one-piece}  translations. In this paper, we use this simplified notion of interpretation.\footnote{For precise
definitions of a
\emph{multi-dimensional interpretation}, an
\emph{interpretability with parameters},  and a
\emph{piece-wise interpretation}, we refer to \cite[pp.10-13]{Visser16}   for more details.}
\smallskip

\begin{definition}[Interpretations II]~
\begin{itemize}
\item A theory $T$ is \emph{interpretable} in a theory $S$ if there exists an
interpretation of $T$ in $S$.
\item Given theories $S$ and $T$, let $S\unlhd T$ denote that $S$ is interpretable in $T$ (or $T$ interprets $S$); let $S\lhd T$ denote that  $T$ interprets $S$ but $S$ does not interpret  $T$.
    \item   We say that a theory $S$ is \emph{strictly weaker} than a theory $T$ w.r.t.~ interpretability if $S\lhd T$.
\item  We say $S$ is a \emph{minimal} RE theory w.r.t.~ interpretability if there is no RE theory $T$ such  that $T\lhd  S$.
\end{itemize}
\end{definition}

The notion of interpretability provides us a method to compare different theories in different languages. If $T$ is interpretable in $S$, then all sentences provable (refutable) in $T$ are mapped, by the interpretation function, to sentences provable (refutable) in $S$.

The theory $\mathbf{VS}$ is introduced by Robert A. Vaught in \cite{Vaught67} (see also \cite{Visser12}). \smallskip

\begin{definition}[The Vaught set theory $\mathbf{VS}$]~\label{def of VS}
The theory $\mathbf{VS}$ is axiomatized by the schema
\[(V_n)\qquad  \forall x_0, \cdots, \forall x_{n-1}  \exists y \forall t (t \in y \leftrightarrow \bigvee_{i<n} t = x_i)\]
for all $n \in\omega$, asserting that $\{x_i: i < n\}$ exists.
\end{definition}
\smallskip

\begin{fact}[\cite{Visser12}, p.383]~\label{fact on VS}
The theory {\sf VS} interprets the theory $\mathbf{R}$ and hence is {\sf EI}.
\end{fact}

\begin{definition}\label{}
Given two RE theories $A$ and $B$, we define the theory $A \oplus B$ as follows. The signature of $A \oplus B$ is a disjoint sum of the signatures of $A$ and $B$ plus a new $0$-ary predicate symbol $P$. The theory $A \oplus B$ is axiomatised by all $P \rightarrow\varphi$ where $\varphi$ is an axiom of $A$, plus $\neg P\rightarrow\psi$ where $\psi$ is an axiom of $B$. We call $A \oplus B$ the interpretability infimum of $A$ and $B$.
\end{definition}

\begin{remark}\label{def of J}
In this paper, we use Janiczak's theory {\sf J} introduced in \cite[p.136]{Janiczak}, which is a theory  in the language with one binary relation symbol $E$ with the
following axioms.
\begin{description}
  \item[{\sf J1}] $E$ is an equivalence relation.
  \item[{\sf J2}] There is at most one equivalence class of size precisely $n$.
  \item[{\sf J3}] There are at least $n$ equivalence classes with at least $n$ elements.\footnote{Our presentation of the theory {\sf J} follows \cite[p.6]{PV}. We include the axiom {\sf J3} to make the proof of the following fact in Theorem \ref{J thm} more easy: over {\sf J}, every sentence is equivalent with a boolean combination of the $A_n$.}
\end{description}
We define $A_n$ to be the sentence: there exists an equivalence class of size precisely
$n + 1$. Note that the $A_n$ are mutually independent over {\sf J}.
\end{remark}

\begin{theorem}~\label{J thm}
\begin{itemize}
  \item {\sf J} is decidable (see \cite[Theorem 4]{Janiczak}).
  \item Over {\sf J}, every sentence is equivalent with a Boolean combination of the $A_n$'s
  (see \cite[Lemma 2]{Janiczak}).
\end{itemize}
\end{theorem}
\smallskip

\begin{definition}[\cite{PV}, p.7]~\label{}
Given $X\subseteq \mathbb{N}$, we say that $W$ is a \emph{${\sf J},X$-theory} when $W$ is axiomatised over {\sf J} by boolean combinations of sentences $A_s$ for $s \in X$.
\end{definition}

Theorem \ref{key corollary} is an important tool we use in proving Theorem \ref{minimal tEI}.
\smallskip
\begin{theorem}[\cite{PV}, Theorem 4.5]~\label{key corollary}
If $U$ is a consistent essentially undecidable RE theory, then we can effectively find an infinite recursive set $X$ (from an index of $U$) such that no consistent ${\sf J}, X$-theory interprets $U$.
\end{theorem}

\section{There are no minimal {\sf tEI} theories}

In this section, we propose {\sf tEI} theories, the theory version of {\sf EI} theories, and prove that there are no minimal {\sf tEI} theories w.r.t. interpretability: for any {\sf tEI} theory $T$, we can effectively find a theory which is {\sf tEI} and strictly weaker than $T$ w.r.t. interpretability. Our proof uses Theorem 4.5 in \cite{PV}. Finally, based on a theorem due to Marian B. Pour-EI, we prove that {\sf tEI} theories are equivalent with {\sf EI} theories. As a corollary, there are no minimal {\sf EI} theories w.r.t. interpretability.

In the definition of {\sf EI} theories, $W_i$ and $W_j$ are arbitrary sets of natural numbers; if we view $W_i$ as a set of sentences, it may not even be consistent.
It is natural to consider the theory version of {\sf EI} theories in which we respectively replace $W_i$ and $W_j$ by two sets of sentences $X$ and $Y$ which respectively have the similar properties of $T_P$ and $T_R$. This is our motivation to propose {\sf tEI} theories: the theory version of {\sf EI} theories.

Note that for any RE theory $T$, $T_P$ and $T_R$ have the following properties:
\begin{enumerate}[(i)]
  \item If $\alpha\in T_P$ and $\alpha\vdash\beta$, then $\beta\in T_P$.
  \item
If $\alpha, \beta\in T_P$, then $\alpha\wedge\beta\in T_P$.
  \item
If $\alpha\in T_R$ and $\beta\vdash \alpha$, then $\beta\in T_R$.
  \item
If $\alpha, \beta\in T_R$, then $\alpha\vee\beta\in T_R$.
\end{enumerate}

\smallskip
\begin{definition}~\label{}
\begin{itemize}
  \item We say a set of sentence $X$ is a filter if the following conditions hold:
\begin{enumerate}[(1)]
  \item If $\alpha\in X$ and $\alpha\vdash\beta$, then $\beta\in X$.
  \item If $\alpha, \beta\in X$, then $\alpha\wedge\beta\in X$.
  \item $\bot\notin X$.
  \end{enumerate}
  \item
We say a set of sentence $Y$ is an ideal if the following conditions hold:
\begin{enumerate}[(A)]
  \item If $\alpha\in Y$ and $\beta\vdash \alpha$, then $\beta\in Y$.
  \item  If $\alpha, \beta\in Y$, then $\alpha\vee\beta\in Y$.
  \item $\top\notin Y$.
\end{enumerate}
\item We say a set $Y$ of sentences is co-consistent if $\{\neg\phi: \phi\in Y\}$ is consistent.
\end{itemize}
\end{definition}

Note that for any consistent RE theory $T$, $T_P$ is a filter and $T_R$ is an ideal. In the theory version of {\sf EI} theories we will define, we respectively replace the $W_i$ and $W_j$ in the definition of {\sf EI} theories with a filter set of sentences and an ideal set of sentences.

\begin{definition}\label{tEI}
We say an RE theory $T$ is {\sf tEI} if there exists a  recursive function $f$ such that if $T_P\subseteq X, T_R\subseteq Y$ and $X\cap Y=\emptyset$, where $X$ is a  filter set of sentence and is RE with index $i$, and $Y$ is  an ideal set of sentences and is RE with index $j$, then $f(i,j)$ converges and outputs a sentence neither in $X$ nor in $Y$ (i.e., $f(i,j)\notin X\cup Y$).
\end{definition}

Note that in Definition \ref{tEI}, $X$ is consistent and $Y$ is co-consistent.
We first show that {\sf tEI} theories are closed under interpretability infimum.

\begin{theorem}\label{tEI close}
If $U$ and $V$ are {\sf tEI} theories, then $T=U \oplus V$ is {\sf tEI}.
\end{theorem}
\begin{proof}\label{}

Suppose $U$ is {\sf tEI} with the witnessing function $f_{U}$ and $V$ is {\sf tEI} with the witnessing function $f_{V}$.
We want to find a recursive function $g$  such that $T=U \oplus V$ is {\sf tEI} with the witnessing function $g$.

Suppose $T_P\subseteq X, T_R\subseteq Y$ and $X\cap Y=\emptyset$, where $X$ is a filter set  of sentences with index $i$ and $Y$ is an ideal set of sentences with index $j$. We describe how to compute $g(i,j)$ such that $g(i,j)$ converges and $g(i,j)\notin X\cup Y$.

Define $Z_0=\{\phi: P\rightarrow\phi\in X\}$, $Z_1=\{\phi: \neg P\rightarrow\phi\in X\}$, $Z_2=\{\phi: P\wedge\phi\in Y\}$ and $Z_3=\{\phi: \neg P\wedge\phi\in Y\}$. Note that $Z_0, Z_1, Z_2$ and $Z_3$ are all RE set of sentences. Suppose $Z_0$ has index $k_0$ and $Z_1$ has index $k_1$. Note that $k_0$ and $k_1$ can be computed effectively from $i$.
Suppose $Z_2$ has index $k_2$ and $Z_3$ has index $k_3$. Note that $k_2$ and $k_3$ can be computed effectively from $j$.

Note that $U_P\subseteq Z_0, V_P\subseteq Z_1, U_R\subseteq Z_2$ and $V_R\subseteq Z_3$. Since $X$ has the closure property (1)-(2) and $Y$ has the closure property (A)-(B), $Z_0$ and $Z_1$ also have the closure property (1)-(2), and $Z_2$ and $Z_3$ also have the closure property (A)-(B).

\begin{claim}
Either $Z_0\cap Z_2=\emptyset$ or $Z_1\cap Z_3=\emptyset$.
\end{claim}
\begin{proof}\label{}
Suppose $Z_0\cap Z_2\neq\emptyset$ and $Z_1\cap Z_3\neq\emptyset$. Take $\phi$ and $\psi$ such that
$P\rightarrow\phi\in X$, $P\wedge\phi\in Y$, $\neg P\rightarrow\psi\in X$ and $\neg P\wedge\psi\in Y$.

From $P\rightarrow\phi\in X$ and $\neg P\rightarrow\psi\in X$, by the property (1)-(2), we have $(P\wedge \phi)\vee (\neg P\wedge\psi)\in X$.

From $P\wedge\phi\in Y$ and $\neg P\wedge\psi\in Y$, by the property (B), we have $(P\wedge \phi)\vee (\neg P\wedge\psi)\in Y$.

Thus, $(P\wedge \phi)\vee (\neg P\wedge\psi)\in X\cap Y$ which contradicts that $X\cap Y=\emptyset$.
\end{proof}

Suppose $Z_0\cap Z_2=\emptyset$. Then $\bot \notin Z_0$ (if not, then any formula is in $Z_0$ by the property (1)), and $\top \notin Z_2$ (if not, then any formula is in $Z_2$ by the property $(A)$). Thus, $Z_0$ is a filter and $Z_2$ is an idea. Then  $f_{U}(k_0, k_2)$ is defined and $f_{U}(k_0, k_2)\notin Z_0\cup Z_2$.
Suppose $Z_1\cap Z_3=\emptyset$. Then by a similar argument, $f_{V}(k_1, k_3)$ is defined and $f_{V}(k_1, k_3)\notin Z_1\cup Z_3$.

We now check in stages simultaneously whether $f_{U}(k_0, k_2)$ converges, $f_{V}(k_1, k_3)$ converges, whether $Z_0\cap Z_2=\emptyset$ (i.e. whether we can find $\phi$ such that $P\rightarrow\phi\in X$ and $P\wedge\phi\in Y$) and whether $Z_1\cap Z_3=\emptyset$ (i.e., whether we can find $\psi$ such that $\neg P\rightarrow\psi\in X$ and $\neg P\wedge\psi\in Y$).

Suppose at stage $n$ the procedure is still running and we find at that stage:
\begin{enumerate}[(a)]
  \item If both $f_{U}(k_0, k_2)$ and $f_{V}(k_1, k_3)$ converge, say to values $\theta$ and $\tau$,
        then we define $g(i,j)$ outputs $(P\wedge \theta)\vee (\neg P\wedge \tau)$.
  \item If (a) does not apply and $Z_0\cap Z_2\neq\emptyset$. Then $Z_1\cap Z_3=\emptyset$. We output $f_{V}(k_1, k_3)$  for $g(i,j)$.
      \item  If (a) and (b) do not apply and $Z_1\cap Z_3\neq\emptyset$. Then $Z_0\cap Z_2=\emptyset$. We output $f_{U}(k_0, k_2)$ for $g(i,j)$.
\end{enumerate}
If any of $(a,b,c)$ applied at stage $n$, we stop the procedure. If neither of $(a,b,c)$ happens at stage $n$, we proceed to stage $n+1$.

Since $X\cap Y=\emptyset$, either $Z_0\cap Z_2=\emptyset$ or $Z_1\cap Z_3=\emptyset$. If both are empty, (a) will obtain at some stage.
If one is non-empty, one of the cases (b,c) will obtain at some stage. So $g(i,j)$ will converge at some stage.

\begin{claim}
If $g(i,j)$ has its value via (b) or (c), then $g(i,j)\notin X\cup Y$.
\end{claim}
\begin{proof}\label{}
Suppose $g(i,j)$ received its value via (b). Then $g(i,j)=f_{V}(k_1, k_3)=\tau$. We know $\neg P\rightarrow\tau\notin X$ and $\neg P\wedge\tau\notin Y$. If $\tau\in X$, then by property (1), $\neg P\rightarrow\tau\in X$ which leads to a contradiction. If $\tau\in Y$, then by property (A), $\neg P\wedge\tau\in Y$ which leads to a contradiction. So $g(i,j)\notin X\cup Y$.

By the similar argument, if $g(i,j)$ received its value via (c), then $g(i,j)\notin X\cup Y$.
\end{proof}

\begin{claim}
If $g(i,j)$ has its value via (a), then $g(i,j)\notin X\cup Y$.
\end{claim}
\begin{proof}\label{}
If $g(i,j)$ has its value via (a), then either $Z_0\cap Z_2=\emptyset$ or $Z_1\cap Z_3=\emptyset$.

Suppose $Z_0\cap Z_2=\emptyset$. Then $\theta \notin  Z_0\cup Z_2$, i.e. $P\rightarrow \theta\notin X$ and $P\wedge \theta\notin Y$. We show  $g(i,j)=(P\wedge \theta)\vee (\neg P\wedge \tau)\notin X\cup Y$.  Suppose $(P\wedge \theta)\vee (\neg P\wedge \tau)\in X$. By the property (1), $P\rightarrow \theta\in X$ which leads to a contradiction.
Suppose $(P\wedge \theta)\vee (\neg P\wedge \tau)\in Y$. By the property (A), $P\wedge \theta\in Y$ which leads to a contradiction.

Suppose $Z_1\cap Z_3=\emptyset$. Then $\tau \notin  Z_1\cup Z_3$, i.e. $\neg P\rightarrow\tau\notin X$ and $\neg P\wedge\tau\notin Y$.
We show  $g(i,j)=(P\wedge \theta)\vee (\neg P\wedge \tau)\notin X\cup Y$.  Suppose $(P\wedge \theta)\vee (\neg P\wedge \tau)\in X$, by the property (1), $\neg P\rightarrow\tau\in X$ which leads to a contradiction.
Suppose $(P\wedge \theta)\vee (\neg P\wedge \tau)\in Y$. By the property (A), $\neg P\wedge\tau\in Y$ which leads to a contradiction.
\end{proof}
\end{proof}

\begin{remark}
In the proof of Theorem \ref{tEI close}, the argument that $g(i,j)$ will converge at some stage is not effective. But the description of the  function $g$ is effective.
\end{remark}

\begin{lemma}\label{tEI is EU}
If $T$ is ${\sf tEI}$, then $T$ is  ${\sf EU}$.
\end{lemma}
\begin{proof}\label{}
Suppose $T$ is ${\sf tEI}$ with the witnessing function $f$. To show  $T$ is  ${\sf EU}$, by Lemma \ref{relation about EI}(4), it suffices to show that $T$ is essentially incomplete. Let $S$ be a consistent RE extension  of $T$, $S_P$ has index $i$ and $S_R$ has index $j$. Note that $T_P\subseteq S_P$ and $T_R\subseteq S_R$. Clearly, $S_P$ is a filter and $S_R$ is an ideal. Then $f(i,j)$ converges and $f(i,j)\notin S_P\cup S_R$. That is, $f(i,j)$ outputs a sentence independent of $S$. Thus, $S$ is incomplete.
\end{proof}
\smallskip

\begin{theorem}[The $s$-$m$-$n$ theorem, \cite{Rogers87}, p.23]~\label{}
For any $m, n\geq 1$, there exists a recursive function $s_n^m$ of $m+1$ variables such that for all $x, y_1, \cdots, y_m, z_1, \cdots, z_n$, we have
\[\phi^n_{s_n^m(x, y_1, \cdots, y_m)}(z_1, \cdots, z_n)=\phi_{x}^{m+n}(y_1, \cdots, y_m, z_1, \cdots, z_n).\]
\end{theorem}

\begin{lemma}\label{EI lemma}
Suppose $(Y,Z)$ is an effectively inseparable pair of RE sets.
Define the theory $V={\sf J} + \{A_n: n \in Y\}+ \{\neg A_n: n \in Z\}$, where $A_n$ is the sentence defined in Remark \ref{def of J}. Then $V$ is {\sf EI}.
\end{lemma}
\begin{proof}\label{}

We want to find a recursive function $h(i,j)$ such that if $V_P\subseteq W_i, V_R\subseteq W_j$ and $W_i\cap W_j=\emptyset$, then $h(i,j)$ converges and $h(i,j)\notin W_i\cup W_j$.

Define the function $f: n\mapsto A_n$. Clearly, $f$ is a total recursive function. By s-m-n theorem, there is a recursive function $g$ such that $f^{-1}[W_i]=W_{g(i)}$. Suppose $(Y,Z)$ is effectively inseparable via the recursive function $t(i,j)$. Define $h(i,j)=f(t(g(i), g(j)))$. Clearly, $h$ is   recursive.

Suppose
$V_P\subseteq W_i, V_R\subseteq W_j$ and $W_i\cap W_j=\emptyset$.
Note that $Y\subseteq f^{-1}[V_P]\subseteq f^{-1}[W_i]=W_{g(i)}$, and $Z\subseteq f^{-1}[V_R]\subseteq f^{-1}[W_j]=W_{g(j)}$.
Note that $W_{g(i)}\cap W_{g(j)}=\emptyset$ since $W_i\cap W_j=\emptyset$.
Since $(Y,Z)$ is {\sf EI} via the function $t$, we have $t(g(i), g(j))$ converges and $t(g(i), g(j))\notin W_{g(i)}\cup W_{g(j)}$. Then, $h(i,j)$ converges and $h(i,j)\notin W_i\cup W_j$. Thus, $V$ is {\sf EI}.
\end{proof}

\begin{theorem}\label{minimal tEI}
There are no minimal {\sf tEI} theories w.r.t. interpretation: for any {\sf tEI} theory $U$, we can effectively find a theory which is {\sf tEI} and strictly weaker than $U$ w.r.t. interpretation.
\end{theorem}
\begin{proof}\label{}
Let $U$ be an {\sf tEI} theory. By Lemma \ref{tEI is EU}, {\sf tEI} theories are essentially undecidable. Theorem \ref{key corollary} applies to {\sf tEI} theories. Thus, we can effectively find an infinite recursive set $X$ (from an index of $U$) such that no consistent
${\sf J},X$-theory interprets $U$. Let $(Y, Z)$ be an {\sf EI} pair of RE sets which are subsets of $X$.
Define the theory $V$ as $V={\sf J} + \{A_n: n \in Y\}+ \{\neg A_n: n \in Z\}$.
By Lemma \ref{EI lemma}, $V$ is {\sf EI} and hence is {\sf tEI}. Let $T=U\oplus V$. Note that from a given {\sf tEI} theory $U$, we can effectively find such a theory $T$.
By Theorem \ref{tEI close}, $T$ is {\sf tEI}. Since no consistent
${\sf J},X$-theory interprets $U$, $U$ is not interpretable in $V$.
Thus, $T\lhd U$ since $U$ is not interpretable in $T$.
\end{proof}

\begin{remark}
Theorem \ref{minimal tEI} does not say that the proof of it is effective, and it only says that we can effectively find such a theory $T$ from $U$. In fact, the proof that $T$ has the claimed properties is not effective.
\end{remark}

\begin{remark}
The proof of Theorem \ref{minimal tEI} uses the theory {\sf J} of one equivalence relation. Some reader may think that {\sf J} is not an arithmetic theory and not natural enough. In \cite{PV}, the theory ${\sf Succ}^{\circ}$, a  theory in the language of arithmetic about zero and successor, is introduced.
We could also use the theory ${\sf Succ}^{\circ}$ instead of {\sf J} in the proof of Theorem \ref{key corollary} and Theorem \ref{minimal tEI} since they have the same key properties we need (see \cite[p.5, p.13]{PV}).
\end{remark}

Based on Pour-EI's Theorem \ref{Symullyan thm}, we  prove that {\sf tEI} is equivalent with {\sf EI}.

\begin{theorem}\label{EI is tEI}
For any consistent RE theory $T$, the following statements are equivalent:
\begin{enumerate}[(1)]
  \item $T$ is {\sf EI};
  \item $T$ is {\sf tEI};
  \item $T$ is {\sf EET}.
\end{enumerate}
\end{theorem}
\begin{proof}\label{}
Clearly, {\sf EI} implies {\sf tEI}.
We show that {\sf tEI} implies {\sf EET}. Suppose $T$ is {\sf tEI} with the witnessing recursive function $f$. There exists a recursive function $h$ such that $\{\phi: \neg\phi\in W_i\}=W_{h(i)}$. Define $g(i)=f(i, h(i))$. Clearly, $g$ is recursive. For any $i$, if $S$ is a consistent RE extension of $T$ with index $i$, then $h(i)$ is the index of $S_R$. Note that $T_P\subseteq S_P, T_R\subseteq S_R, S_P$ is a filter and $S_R$ is an ideal.  Then $f(i, h(i))$ converges and $f(i, h(i))\notin S_P\cup S_R$. Thus, $g(i)$ outputs a sentence independent of $S$. Hence, $T$ is {\sf EET} with the witnessing function $g$.
Finally, by Theorem \ref{Symullyan thm}, {\sf EET} implies {\sf EI}.
\end{proof}

As a corollary of Theorem \ref{minimal tEI} and Theorem \ref{EI is tEI}, we have:

\begin{theorem}\label{minimal EI}
There are no minimal {\sf EI} theories w.r.t. interpretability: for any {\sf EI} theory $T$, we can effectively find a theory which is {\sf EI} and strictly weaker than $T$ w.r.t. interpretability.
\end{theorem}

\section{There are no minimal finitely axiomatizable {\sf EI} theories}

In this section, we prove that there are no minimal finitely axiomatizable {\sf EI} theories w.r.t. interpretability. We give two proofs of this result in Theorem \ref{effective version for finte theory} and Theorem \ref{no minial finite EI}. Both proofs use the properties of the theory {\sf TN} (see Definition \ref{TN}) and Theorem \ref{close EI} whose proof uses Smullyan's Theorem \ref{EI thm}.
Theorem \ref{effective version for finte theory} assumes piecewise interpretations and we can show that given a finitely axiomatizable {\sf EI} theory $T$, we can effectively find a finitely axiomatizable {\sf EI} theory, which is strictly weaker than $T$ w.r.t. interpretability.
One main tool of Theorem \ref{effective version for finte theory} is   Harvey Friedman's Theorem \ref{dense lemma}. The proof of Theorem \ref{no minial finite EI} does not assume piecewise interpretations, and uses an argument in the proof of Theorem 3.1 in \cite{PV}.

We first introduce the theory {\sf TN} of numbers.\smallskip

\begin{definition}[The theory {\sf TN}, \cite{Visser16}, p.4]~\label{TN}
The theory {\sf TN} consists of the following axioms:
\begin{description}
  \item[{\sf TN1}] $\vdash x\nless \mathbf{0}$;
  \item[{\sf TN2}] $\vdash (x<y\wedge y<z)\rightarrow x<z$;
  \item[{\sf TN3}] $\vdash x<y\vee x=y\vee y<x$;
  \item[{\sf TN4}] $\vdash x=\mathbf{0}\vee \exists y (x=\mathbf{S}y)$;
  \item[{\sf TN5}] $\vdash \mathbf{S} x\nless x$;
  \item[{\sf TN6}] $\vdash x<y\rightarrow (x< \mathbf{S} x\wedge y\nless \mathbf{S} x)$;
  \item[{\sf TN7}] $\vdash  x + \mathbf{0} = x$;
  \item[{\sf TN8}] $\vdash x+\mathbf{S} y=\mathbf{S} (x+y)$;
  \item[{\sf TN9}] $\vdash x \times\mathbf{0} = \mathbf{0}$;
  \item[{\sf TN10}] $\vdash x\times \mathbf{S} y=x\times y+x$.
\end{description}
\end{definition}
For more details about the theory {\sf TN}, see \cite[p.4]{Visser16}.
A $\Delta^0_0$-formula is pure if all bounding terms are variables and all occurrences of terms are in sub-formulas of the form $\mathbf{S}x=y, x+y=z$ and $x\times y=z$. A $\Sigma^0_1$-sentence is pure if it is equivalent with the form $\exists x \phi(x)$ where $\phi$ is a pure $\Delta^0_0$-formula. We can transform any $\Sigma^0_1$-sentence into a pure $\Sigma^0_1$-sentence. In this section, we assume that all $\Sigma^0_1$-sentences are rewritten in pure form.

Let $\psi=\exists x\phi(x)$, where $\phi$ is a pure $\Delta^0_0$-formula.
Define  the finitely axiomatized theory $[\psi]$ as follows:
\[[\psi]=\mathbf{TN}+ \exists x\exists y<x \, \phi(y).\]
\smallskip

\begin{definition}[\cite{Visser16}, p.4]~\label{}
Suppose $\varphi=\exists x\, A(x)$ and $\psi=\exists x\, B(x)$ are two $\Sigma^0_1$-sentences.
We Define:
\begin{enumerate}[(1)]
  \item $\varphi\preceq\psi\triangleq \exists x (A(x)\wedge\forall y< x \,\neg B(y))$;
  \item $\varphi\prec\psi\triangleq \exists x (A(x)\wedge\forall y\leq x \, \neg B(y))$;
  \item If $\theta$ is $\varphi\preceq\psi$, then $\theta^{\perp}=\psi\prec \varphi$;
  \item If $\theta$ is $\varphi\prec\psi$, then $\theta^{\perp}=\psi\preceq \varphi$.
\end{enumerate}

\end{definition}
\smallskip

\begin{fact}[\cite{Visser16}, Theorem 1; \cite{PV}, p.5]~\label{key fact on Q}
Suppose $\varphi, \psi$ are $\Sigma^0_1$-sentences.
\begin{enumerate}[(1)]
  \item If $\psi$ does not hold, then $[\psi] \supseteq \mathbf{R}$.
  \item If $\psi$ holds and we allow piecewise interpretations, then $[\psi]\unlhd \top$.\footnote{For the notion of piecewise interpretability, we refer to \cite[p.13]{Visser16}.}
  \item If $\varphi\preceq\psi$, then $[\psi]\vdash \varphi$.
  \item Let $A=\varphi\preceq\psi$. If $\varphi$ (or $\psi$) holds, then either $A$ holds or $A^{\perp}$ holds.
\end{enumerate}
\end{fact}

Now we prove that {\sf EI} theories are closed under interpretability infimum. An important  tool we use is Smullyan's Theorem \ref{EI thm}.

\begin{definition}\label{}
Let $(A,B)$ and $(C,D)$ be  disjoint pairs of RE sets.
We say $(A,B)$ is \emph{semi-reducible} to $(C,D)$ if there is a recursive function $f(x)$ such that $f(x)\in C$ if $x\in A$,  and $f(x)\in D$ if $x\in B$.
\end{definition}

\begin{theorem}[\cite{Smullyan}, pp.70-126]~\label{EI thm}
For any RE theory $T$, $T$ is {\sf EI} iff any disjoint pair $(A,B)$ of RE sets is semi-reducible to $(T_P, T_R)$.
\end{theorem}

\begin{theorem}\label{close EI}
If $U$ and $V$ are  {\sf EI} theories, then $T=U \oplus V$ is  {\sf EI}.
\end{theorem}
\begin{proof}\label{}
Suppose $U$ and $V$ are {\sf EI} theories, and $T=U \oplus V$. By Theorem \ref{EI thm}, it suffices to show that any disjoint pair of RE sets is semi-reducible to $(T_P, T_R)$.

Let $(A,B)$ be any disjoint pair of RE sets. Since $U$ is {\sf EI}, there is a recursive function $f_1$ such that:
\begin{enumerate}[(i)]
  \item if $n\in A$, then $f_1(n)\in U_P$;
  \item if $n\in B$, then $f_1(n)\in U_R$.
\end{enumerate}
Since $V$ is {\sf EI}, there is a recursive function $f_2$ such that:
\begin{enumerate}[(i)]
  \item if $n\in A$, then $f_2(n)\in V_P$;
  \item if $n\in B$, then $f_2(n)\in V_R$.
\end{enumerate}

Note that $U_P\subseteq \{\phi: P\rightarrow\phi\in T_P\}$,
$V_P\subseteq \{\phi: \neg P\rightarrow\phi\in T_P\}$,
$U_R\subseteq \{\phi: P\wedge\phi\in T_R\}$,
$V_R\subseteq \{\phi: \neg P\wedge\phi\in T_R\}$.
Define $g(n)=(P\rightarrow f_1(n))\wedge (\neg P\rightarrow f_2(n))$. Since $f_1$ and $f_2$ are recursive,
$g$ is also recursive.

Suppose $n\in A$. Then $P\rightarrow f_1(n)\in T_P$ and $\neg P\rightarrow f_2(n)\in T_P$.
Thus, if $n\in A$, then $g(n)\in T_P$.

Suppose $n\in B$. Since $P\wedge f_1(n)\in T_R$, we have:

\begin{equation}\label{}
T\vdash P\rightarrow \neg f_1(n).\label{first eqn}
\end{equation}

Since $\neg P\wedge f_2(n)\in T_R$, we have:

\begin{equation}\label{}
T\vdash \neg P\rightarrow \neg f_2(n).\label{second eqn}
\end{equation}

From (\ref{first eqn}) and (\ref{second eqn}), we have:
$T\vdash (P\wedge \neg f_1(n))\vee (\neg P\wedge \neg f_2(n))$.

Thus, if $n\in B$, then $g(n)\in T_R$. Hence, $(A,B)$ is semi-reducible to $(T_P, T_R)$.
\end{proof}

\begin{remark}
By the similar argument as in Theorem \ref{minimal tEI}, Theorem \ref{close EI} can also be used to give a second proof of Theorem \ref{minimal EI}.
\end{remark}

\begin{theorem}[Harvey Friedman]\label{dense lemma}
Assuming that we allow piecewise interpretations, for any finitely axiomatizable theory $A$, if $\top\lhd A$, then there exists a finitely axiomatizable theory $B$ such that
$\top\lhd B\lhd A$.\footnote{This proof  is simple than Friedman's proof in \cite{Friedman07}, and the idea of this proof is from \cite{Visser16}.}
\end{theorem}
\begin{proof}\label{}
We employ the G\"{o}del fixed point construction to find a sentence $\theta$ such that $\mathbf{PA}\vdash \theta\leftrightarrow (A\unlhd A\oplus [\theta])\preceq (A\oplus [\theta]\unlhd \top)$.

\begin{claim}
$A\ntrianglelefteq A\oplus [\theta]$.
\end{claim}
\begin{proof}\label{}
Suppose $A\unlhd A\oplus [\theta]$ holds. By Fact \ref{key fact on Q}(4), either $\theta$ holds or $\theta^{\perp}$  holds.

Case one: Suppose $\theta$ holds. By Fact \ref{key fact on Q}(2),  $[\theta]\unlhd \top$. Since $A\unlhd A\oplus [\theta]\unlhd [\theta]$, we have $A\unlhd \top$, which is a contradiction.

Case two: Suppose $\theta$ does not hold. Then $\theta^{\perp}$  holds. By definitions, it is easy to check that $\theta^{\perp}\preceq \theta$ holds. By Fact \ref{key fact on Q}(3), $[\theta]\vdash \theta^{\perp}$. Since $[\theta]\vdash \theta^{\perp}\wedge \theta$ and $\theta^{\perp}\wedge \theta\vdash \bot$, we have $[\theta]\vdash \bot$. Since $(A\unlhd A\oplus [\theta])\preceq (A\oplus [\theta]\unlhd \top)$
does not hold and $A\unlhd A\oplus [\theta]$ holds, we have $A\oplus [\theta]\unlhd \top$ holds. Since $[\theta]\vdash \bot$, we have $A\oplus [\theta]=A$. Thus $A\unlhd \top$, which is a contradiction.
\end{proof}

By the similar argument, we can show that   $A\oplus [\theta]\ntrianglelefteq \top$.
Thus, $\top\lhd A\oplus [\theta]\lhd A$.
\end{proof}

\begin{theorem}\label{effective version for finte theory}
Assuming that we allow piecewise interpretations, if $T$ is a finitely axiomatizable {\sf EI} theory, then we can effectively find a finitely axiomatized {\sf EI} theory which is strictly weaker than $T$ w.r.t. interpretability.
\end{theorem}
\begin{proof}\label{}
This follows from Theorem \ref{dense lemma}. Let $\theta$ be the sentence defined in Theorem \ref{dense lemma}. Note that since  $T$ and $[\theta]$ are finitely axiomatizable, the theory $S=T\oplus [\theta]$ as in Theorem \ref{dense lemma}  is also finitely axiomatizable.
By Theorem \ref{dense lemma}, $\top\lhd S\lhd T$. Recall that the fixed point construction of the sentence $\theta$ in the proof of Theorem \ref{dense lemma} is effective.  Thus, we can effectively find such a theory $S$. Suppose $\theta$ holds. Since we allow piecewise interpretations, by Fact \ref{key fact on Q}(2), $[\theta]\unlhd \top$. Thus, $S\unlhd \top$, which contradicts that $\top\lhd S$. Hence, $\theta$ does not hold. By Fact \ref{key fact on Q}(1),  $[\theta]\supseteq \mathbf{R}$.
By Lemma \ref{relation about EI}(8), $[\theta]$ is {\sf EI}. By Theorem \ref{close EI}, $S$ is {\sf EI}.
\end{proof}

Now, we give another proof which does not assume piecewise interpretability, but its proof is non-constructive in some sense.

\begin{lemma}\label{minimum EI}
If there is a minimal finitely axiomatizable {\sf EI} theory w.r.t. interpretability, then it is a minimum  finitely axiomatizable {\sf EI} theory w.r.t. interpretability.
\end{lemma}
\begin{proof}\label{}
Suppose $T$ is a minimal finitely axiomatizable {\sf EI} theory w.r.t. interpretability.  Suppose $S$ is any  finitely axiomatizable {\sf EI} theory. Consider the theory $T\oplus S$. By Theorem \ref{close EI}, $T\oplus S$ is a finitely axiomatizable  {\sf EI} theory. By the minimality of $T, T\unlhd T\oplus S$ and thus $T\unlhd S$. Hence, $T$ is the minimum  finitely axiomatizable {\sf EI} theory w.r.t. interpretability.
\end{proof}

\begin{theorem}\label{no minial finite EI}
There are no minimal finitely axiomatizable {\sf EI} theories w.r.t. interpretability.
\end{theorem}
\begin{proof}\label{}
By Theorem \ref{close EI}, if $S$ and $T$ are finitely axiomatizable {\sf EI} theories, then $S \oplus T$ is also a finitely axiomatizable {\sf EI} theory.
By Lemma \ref{minimum EI}, it suffices to show that there is no minimum  finitely axiomatizable {\sf EI} theory w.r.t. interpretability. The following argument comes from Theorem 3.1 in \cite{PV}. The key observation is Lemma \ref{relation about EI}(8).

Suppose $T$ is the minimum  finitely axiomatizable {\sf EI} theory w.r.t. interpretability.
Consider any $\Sigma^0_1$-sentence $\sigma$. If $\sigma$ holds, then $[\sigma]$ has a finite model and thus $T\ntrianglelefteq [\sigma]$. If $\sigma$ does not hold and $[\sigma]$ is consistent, since $[\sigma]\supseteq \mathbf{R}$ by Fact \ref{key fact on Q}(1), $[\sigma]$ is {\sf EI} by Lemma \ref{relation about EI}(8), and thus $T\unlhd [\sigma]$ since $T$ is the minimum  finitely axiomatizable {\sf EI} theory. If $[\sigma]$ does not hold and is inconsistent, then $T\unlhd [\sigma]$. Thus,  $\sigma$  does not hold if and only if $T\unlhd [\sigma]$, which contradicts the fact that the set of $\sigma$ such that $T\unlhd [\sigma]$	 is an RE set.
\end{proof}

\begin{remark}
The proof of Theorem \ref{no minial finite EI} does not tell us, given a finitely axiomatizable {\sf EI} theory $T$, how to construct another  finitely axiomatizable {\sf EI} theory which is strictly weaker than $T$ w.r.t. interpretability. It only tells us that if there is a minimal finitely axiomatizable {\sf EI} theory w.r.t. interpretability, this leads to a contradiction. In this sense, the proof of Theorem \ref{no minial finite EI} is non-constructive.
\end{remark}

We conclude the paper with some perspectives for future work. The work in this paper can be extended from two perspectives. We could examine the existence of minimal RE theories with meta-mathematical properties with respect to other notions of reducibility except for Turing degree and interpretability that we have considered in this paper. On the other hand, we could examine the existence of minimal RE theories  w.r.t. interpretability for other meta-mathematical properties that we have not examined. For example, a natural question is: are there minimal {\sf Creative} theories w.r.t. interpretability. We did not explore it here.


\begin{thebibliography}{100}

\bibitem{Cheng23}
Yong Cheng. On the relationships between some meta-mathematical properties of arithmetical theories.  \emph{Logic Journal of the IGPL}, in press, DOI: 10.1093/jigpal/jzad015,  2023.


\bibitem{Friedman07}
Harvey M. Friedman. Interpretations: according to Tarski. Nineteenth Annual Tarski Lectures, 2007.

\bibitem{Janiczak} Antoni Janiczak. Undecidability of some simple formalized theories.
\emph{Fundamenta  Mathematicae}, vol. 40 (1953), pp. 131-139.


\bibitem{Kleene}
Stephen C. Kleene. \emph{Introduction to Metamathematics}. Amsterdam, Gr\"{o}ningen,
New York and Toronto, 1952.

\bibitem{PV} Fedor Pakhomov, Juvenal Murwanashyaka and Albert Visser. There are no minimal essentially undecidable Theories. 	 \emph{Journal of Logic and Computation}, DOI: 10.1093/logcom/exad005


\bibitem{Pour-EI}
Marian Boykan Pour-EI. Effectively Extensible Theories. \emph{The Journal of Symbolic Logic}, Mar, 1968, Vol. 33, No.1, pp. 56-68.


\bibitem{Rogers87} Hartley Rogers. \emph{Theory of Recursive Functions and Effective Computability}. MIT Press, Cambridge, 1987.

\bibitem{Shoenfield61}
Joseph R. Shoenfield. Undecidable and creative theories. \emph{Fund. Math.} 48, 171-179, 1961.


\bibitem{Smullyan}
Raymond M.Smullyan. \emph{Recursion theory for meta-mathematics}. Oxford University Press, 1993.



\bibitem{undecidable} {Alfred Tarski, Andrzej Mostowski and Raphael M. Robinson.}
\emph{Undecidable theories}. Studies in Logic and the Foundations of Mathematics, North-Holland, Amsterdam, 1953.

\bibitem{Vaught67} Robert A. Vaught. Axiomatizability by a schema. \emph{The Journal of Symbolic},  no. 4, pp.473-479, 1967.

\bibitem{Visser12}
Albert Visser. Vaught's theorem on axiomatizability by a scheme.
\emph{The Bulletin of Symbolic Logic}, Volume 18, Number 3, 2012.

\bibitem{Visser16} {Albert Visser.}
On $\mathbf{Q}$.
\emph{Soft Comput}, 21(1): 39-56 (2017).





\end{thebibliography}
\end{document}